\documentclass[12pt,english,nocouv,nofiligrane]{amsart}
\usepackage{array}
\usepackage{threeparttable}
\usepackage{rotating}
\usepackage[section]{placeins}
\usepackage{amssymb} 
\usepackage{mathrsfs}
\usepackage{enumerate}
\usepackage{color}
\usepackage{amsrefs}        
\usepackage[all]{xy}

\input xy
\xyoption{all}
\newtheorem{lemma}{Lemma}[section]
\newtheorem{lemm}[lemma]{Lemma}
\newtheorem{prop}[lemma]{Proposition}
\newtheorem{coro}[lemma]{Corollary}
\newtheorem*{theomain*}{Main Theorem}
\newtheorem*{theoA*}{Theorem A}
\newtheorem*{theoB*}{Theorem B}

\newtheorem{rema}[lemma]{Remark}

\newtheorem{defi}[lemma]{Definition}
\newtheorem{exam}[lemma]{Example}
\newcommand{\N}{\mathbb N}
\newcommand{\Z}{\mathbb Z}

\newcommand{\Q}{\mathbb Q}
\newcommand{\C}{\mathbb C} 
\newcommand{\K}{\mathbb K}

\begin{document} 
\title{On the Idempotent Conjecture 
for Sidki Doubles}
\author{Indira Chatterji}
 \address{LJAD, UMR 7351 CNRS, Universit\'e de Nice}  
\email{indira@unice.fr}
\author{Guido Mislin}
\address{Department of Mathematics ETHZ, Z\"urich}    
 \email{mislin@math.ethz.ch}
     \date{September 19th, 2022}
     \dedicatory{To Ruth Charney}
\begin{abstract} 
Let $G$ be a group and $X(G)$ its Sidki Double. The idempotent conjecture says that there should be no non-trivial idempotent in the complex group ring of a torsion-free group. We investigate this conjecture for the Sidki double of a torsion-free group, and obtain a partial result.
\end{abstract}
\maketitle
\section*{Introduction}

Let $G$ be a group and $\psi: G\to G^\psi$ an isomorphism
between (disjoint) groups.
Write $g^\psi$ for $\psi(g)$ and denote by $[x,y]$ for the commutator
$x^{-1}y^{-1}xy$ of two group elements $x$ and $y$.
The {\it Sidki Double $X(G)$ of $G$}
is the group defined by the following presentation:
$$ X(G)=\left<G, G^\psi\,|\, [g,g^\psi]\,, g\in G\right>\,.$$
From the definition we see that $G\mapsto X(G)$ is a functor
from groups to groups. Also, $X(G)$ maps onto $G\times G$ and is a
 factor group of the free product
$G*G$ :
$$G*G \longrightarrow X(G)\longrightarrow G\times G\,.$$
\noindent
It follows that we can view $G$ as well as $G^\psi$ as subgroups of $X(G)$.
We then write $g,g^\psi\in X(G)$ for $g\in G$, if no confusion is possible.
As elements of $X(G)$, $g$ and $g^\psi$ commute but are distinct, if $g\neq e$.
The doubled group $X(G)$ inherits many properties from the original group $G$. 
For instance, if $G$ is solvable then
$X(G)$ is solvable too (cf.\,Sidki \cite{sidki}, Corollary 4.8). More generally, it is straightforward to see (cf.\ Section \ref{examples}) that for $G$ amenable, then $X(G)$ is amenable as well. In this note, we investigate the stability of various idempotent conjectures when taking Sidki doubles.

\medskip

Recall that an idempotent in a unital ring is {called \sl{trivial}} if it is equal to $0$ or $1$. In case where the group $G$ has an element $g\in G$ of $n$-torsion, one checks that $p=\frac{1}{n}(1+g+\cdots+g^{n-1})$ is an idempotent and
Kaplansky conjectured that for a torsion-free group $G$ and $\K$ any field,
the only idempotent of $\K G$ are the trivial ones (cf.\,\cite{Passman}). Two other famous conjectures due to Kaplansky and others (concerning torsion-free groups) are
the \lq\lq Unit Conjecture\rq\rq, stating that all units of $\K G$
are of the form $\lambda g$ with $\lambda\in\K$ and $g\in G$,
and the \lq\lq Zero-Divisor Conjecture\rq\rq, stating that $\K G$ is a domain. Both of these conjectures imply the Idempotent Conjecture.  For
the Unit Conjecture over fields of finite characteristic, there are now counterexamples in characteristic 2 according to Gardam \cite{Gardam} and in general
 prime characteristic according to Murray \cite{Murray}.
 In the sequel, we will mainly deal with the case of $\K=\C$, the field of complex numbers.
 
We are interested in proving that $\C X(G)$ has no non-trivial idempotent in case where $\C G$ has none. Of course, a necessary condition for that to hold is that $X(G)$ is torsion free, and according to \cite{LO} this is in general not the case as for instance $X({\bf Z}^n)$ has 2-torsion as soon as $n\geq 3$. We shall see more examples in Proposition  \ref{TorsionExample}, but we don't know for instance whether $X({\bf F}_2)$ is torsion-free for ${\bf F}_2$ the free group on two generators. Our main theorem is the following.
\begin{theomain*}\label{A}
Let $G$ be a group and assume that its Sidki double $X(G)$ is torsion-free. 
If $\C (G\times G\times G)$ has no non-trivial idempotent, then
$\C X(G)$ does  not contain any idempotent other that $0$ and $1$.
%Trace Conjecture,then $\C G$ does too.
\end{theomain*}
It is not true in general that if $\C G$ satisfies the idempotent conjecture, then so does its Sidki double $X(G)$, since the groups in Proposition \ref{TorsionExample} can be chosen to satisfy the idempotent conjecture.
There are
many classes of groups $G$ for which it is known that $\C(G\times G\times G)$
has no non-trivial idempotent, for instance if $G$ is a torsion-free a-T-menable
group. This follows from the fact that such groups satisfy the Bass Conjecture over ${\C}$,
cf. \cite{BCM}.

The proof of the Main Theorem relies on Kaplansky's Theorem, which states that for any group $H$ and idempotent $x=x^2=\sum x_h\cdot h \in \C H$, then 
$$x\in \{0,1\} \Longleftrightarrow x_e\in\{0,1\}\,. $$
We will refer to the coefficient $x_e$ of $x$ as the Kaplansky Trace of $x$
and denote it by $\kappa(x)$.
One can define $\kappa$ on arbitrary
elements of $\C H$ by putting  
$$\kappa(a)=a_e\,,\quad \text{for any} \quad a=\sum a_h\cdot h\in \C H\,.$$
The function $\kappa: \C H\to\C$ is a {\sl{trace}}, meaning that it
is $\C$-linear and satisfies
$$\kappa(xy)=\kappa(yx)\,,$$
for all $x,y\in\C H$.
We will also make use of another trace function
on $\C H$, the {\sl{augmentation trace}}  $\epsilon:\C H\to\C$
defined by
$$ \epsilon(\sum a_h\cdot h)=\sum a_h\,.$$
Note that $\epsilon (xy)=\epsilon(yx)$ and $\epsilon$ is $\C$-linear and moreover $\epsilon$ is multiplicative: $\epsilon(ab)=\epsilon(a)\epsilon(b)$, making it a ring homomorphism. Therefore, if $x=x^2\in\C H$, then $\epsilon(x)$ is an idempotent
in $\C$, thus $\epsilon(x)\in \{0,1\}$. One concludes that if an idempotent $x$
in $\C H$ satisfies $\kappa(x)=\epsilon(x)$ then $\kappa(x)\in\{0,1\}$
and therefore, by Kaplansky's Theorem, $x$ equals $0$ or $1$. That $\kappa(x)=\epsilon(x)$ when those traces are extended to matrices is the content of the {\bf{Weak Bass Trace Conjecture}}. Indeed, the definition of these traces can be extended to matrices as follows. Let
$A=(a_{ij})\in M_n(\C H)$ with $a_{ij}=\sum_{x\in H}a_{ij}(h)\cdot h$, and $a_{ij}(h)\in\C$ for any $h\in H$. We then set
$$\kappa(A)=\sum_{1\le i\le n}a_{ii}(e)\,,\quad \epsilon(A)=\sum_{h\in H}
(\sum_{1\le i\le n}a_{ii}(h))\,.$$
According to Zaleskii \cite{Zal}, for any idempotent matrix $A$, $\kappa(A)$ is a rational number 
$\ge 0\,,$ and $\epsilon(A)$ is an integer satisfying 
$0\leq \epsilon(A)\leq n$, because $\epsilon(A)$ is the
trace of an idempotent matrix in $M_n(\C)$.

Bass considered in \cite{Bass} the Hattori-Stallings Trace,
defined for an idempotent matrix (or a finitely generated projective $\C H$-module). It is a
class function $H\to \C $ and gives rise to trace functions on idempotent
matrices as follows. If $A=A^2\in M_n(\C H)$ as above, one gets a finitely generated
projective (left) $\C H$-module $P:=(\C H)^n\cdot A$ and for $x \in H$ the
Hattori-Stallings Trace of $A$ is given by
$$r_A(x):= \sum_{y\in [x], 1\le i\le n}a_{ii}(y).$$
Here $[x]$ stands for the conjugacy class of $x$.
The {\bf{Bass Trace Conjecture}} for $\C H$ is the statement that $r_A(x)=0$
for $x\in H$ of infinite order.

The {\bf{Weak Bass Trace Conjecture}} asserts that for a torsion-free
group $H$ and idempotent matrix $A\in M_n(\C H)$
one has $\kappa(A)=\epsilon(A)$, or equivalently, $\sum_{[x], x \neq e } r_A(x)=0$. This can
thus be viewed as a generalization of the Idempotent Conjecture for $\C H$, which is
the case of $(1\times 1)$-matrices.

\smallskip

To prove the Main Theorem we establish
that under the conditions stated, idempotent matrices $A\in M_n(\C X(G))$
satisfy $\kappa(A)=\epsilon(A)$. So in fact we more generally show that if $\C (G\times G\times G)$ satisfies weak Bass trace conjecture, then so does $\C H$ for any torsion-free subgroup $H$ of  $X(G)$, see Proposition \ref{P}. 

\bigskip

In Section 1  we will review some basic properties of Sidki Doubles.
In Section 2 we prove the Main Theorem and in Section 3 we construct 
examples of Sidki Doubles to illustrate some of their features.

\begin{rema}
In \cite{R} Rocco defined a group $\mathcal{V}(G)$, which has features similar
to $X(G)$, as follows (for group elements $x,y$ we write $x^y$ for the conjugate $y^{-1}xy$)
$$\mathcal{V}(G)=\left< G, G^\psi\,|\,[g, h^\psi]^{k^\epsilon}=[g^k,(h^k)^\psi],\,g, h, k\in G ,\epsilon\in\{1,\psi\}\right>.$$
From the definition, one can see that $\mathcal{V}(G)$ fits in between
the free product and the cartesian product of $G$ with itself:
$$ G\ast G\longrightarrow\!\!\!\!\to\mathcal{V}(G)
\longrightarrow\!\!\!\!\to G\times G\,,$$
both maps being surjective. According to Rocco \cite{R}*{Section 2} the group $\mathcal{V}(G)$ maps onto a quotient of $X(G)$, with central kernel $\Delta(G)<\mathcal{V}(G)$, the subgroup $\Delta(G)$
being generated by all $[g,g^\psi]$ with $g\in G$. More precisely, there is a stem-extension (those are discussed in Section \ref{Sstem})
$$\Delta(G)\to \mathcal{V}(G)\to X(G)/R(G)\,,$$
with $R(G)=[G,L(G),G^\psi]<X(G)$ the subgroup generated by triple commutators, where $L(G)=\langle g^{-1}g^\psi| g\in G\rangle<X(G)$.

 \medskip\noindent
 Our methods could also be applied to the study of idempotent elements in
the group ring $\mathbb{C}\, \mathcal{V}(G)$.
\end{rema}
We thank the referee for valuable suggestions and for pointing out the reference
\cite{R}. 
We also thank Martin Bridson for interesting conversations, this paper has been written inspired by his lecture.
%%%%%%%%%%%%%%%%%%%%%%%%%%%%%%%%%%%%%%%
\section{Hattori-Stallings trace on Sidki Doubles}
%%%%%%%%%%%%%%%%%%%%%%%%%%%%%%%%%%%%%%%
We will use the notation introduced by Kochloukova and Sidki
in \cite{KS}. For the Sidki Double $X(G)$, as defined in the Introduction,
there is a natural homomorphism
$$\rho: X(G)\longrightarrow G\times G\times G$$
defined on generators by $\rho(g)=(g,g,1)$ and $\rho(g^\psi)=(1,g,g)$ and whose kernel is denoted by $\ker(\rho)=W(G)$, which is a normal subgroup in $X(G)$, abelian according to Sidki in \cite{sidki}. In this section we prove that the Hattori-Stallings trace for an integral or complex group ring of a Sidki double $X(G)$ is always trivial on the elements of infinite order from the subgroup $W(G)<X(G)$.
 \begin{prop}\label{CCW}
 Let $G$ be a finitely generated group, $R=\Z$ or $\C$ and $A\in M_n(R X(G))$ an
 idempotent matrix. If $y\in W(G) < X(G)$ is any element of infinite order, then
 $r_A(y)=0$.
 \end{prop}
To prove this proposition, we first study the conjugation action of $X(G)$ on $W(G)$.
\begin{lemm}\label{L1}
The conjugation action of $X(G)$ on $W(G)$ factors through an
abelian group $A(G)$, which is finitely generated in case $G$ is finitely generated.
\end{lemm}
\begin{proof} Let $\mu: G\times G\times G \to G, (x,y,z)\mapsto y$
denote the middle projection. Then we have
a short exact sequence
$$\operatorname{ker}{(\mu\rho)} =:L(G)\longrightarrow X(G)\stackrel{\mu\rho}\longrightarrow G\,.
$$
The subgroup $L(G)$ is generated by elements of the form $g^{-1}g^\psi\in X(G)$, where $g\in G$.
Similarly, with 
$$\omega: G\times  G\times G\to G\times G, (x,y,z)\mapsto (x,z)$$
the \lq\lq left-right\rq\rq projection and writing $D(G)=\ker(\omega\rho)$ for the
kernel of $\omega\rho$, one can check that
 $$W(G)=D(G)\cap L(G)\lhd X(G)\,.$$
A crucial property we will use is the fact that $L(G)$ centralizes $D(G)$
(cf. Lemma 4.1.6 (ii) of \cite{sidki}). Therefore, $W(G)$ is central in $L(G)D(G)$:
$$W(G)=L(G)\cap D(G)\xrightarrow{\text{central}}L(G)D(G)
\xrightarrow{\text{normal}} X(G)\,.$$
We first claim that the commutator subgroup $G'$ of $G$ satisfies
$$G'\times G'\times G'<\rho(L(G)D(G))\,.$$
This can be seen from the following equations in $G\times G\times G$:
\begin{eqnarray*}([u,v],e,e)& =&(u^{-1},e,u)(v^{-1},e,v))(uv,e,(uv)^{-1})\\
&=&\rho(u^{-1}u^\psi\cdot v^{-1}v^\psi\cdot uv((uv)^{-1})^\psi)\in \rho(L(G))\,,\end{eqnarray*}
using the fact that $L(G)$ is generated by elements of the form
$g^{-1}g^\psi$, $g\in G$. Similarly $(e,e,[u,v])\in \rho(L(G))$. Because $D(G)$  is generated
by elements of the form $[x,y^\psi]$ with $x,y\in G$, 
 and
\begin{eqnarray*}\rho([x,y^\psi])&=&\rho(x^{-1}(y^\psi)^{-1}xy^\psi)\\
&=&(x^{-1},x^{-1},e)(e,y^{-1},y^{-1})(x,x,e)(e,y,y)=(e,[x,y],e)\,,\end{eqnarray*}
we see that  $(e,[x,y],e)\in \rho(D(G))$. One concludes that $\rho(L(G)D(G))$
contains the commutator subgroup of $G\times G\times G$, finishing our first claim.
Now, since $W(G)$ is central in $L(G)D(G)$, the conjugation action of $X(G)$ on $W(G)$ factors through
$$X(G)/L(G)D(G)\cong \rho(X(G))/\rho(L(G)D(G))=:A(G)\,,$$
which is an abelian group, as  
 $G'\times G'\times G' <\rho(L(G)D(G))$.
If $G$ is finitely
generated, then so is $X(G)$ and its factor group $A(G)$.
\end{proof}
 %%%
We can turn to the proof of Proposition \ref{CCW}, which is now a direct consequence of Bass' theorem {\cite{Bass}*{Theorem 8.1, (c)}}.
  \begin{proof}[Proof of Proposition \ref{CCW}] Take $y\in W(G)$ of
  infinite order.  We want to show that $r_A(y)=0$. Because of Bass' Theorem
  {\cite{Bass}*{Theorem 8.1, (c)}} applied to the
finitely generated projective $R X(G)$-module $P=(R X(G))^n\cdot A$,
 it suffices to show that
  there can be only finitely many primes $p$ such that $y$ is conjugate in
  $X(G)$  to $y^{p^n}$ for some positive integer $n$. 
  Assume that $y$ is conjugate to $y^{p^n}$ and denote the set of such primes by $\Pi$. According to Lemma \ref{L1}, the conjugation  action of $X(G)$ on $W(G)$ factors through an abelian group $A(G)$. Therefore, we can choose for any such prime an automorphism $\phi_p\in A(G)$ for which $\phi_p(y)=y^{p^n}$ and since $y$ has infinite order, the set $\{\phi_p\}$ generates a free abelian subgroup 
in $A(G)$ of rank equal to the cardinality of $\Pi$ in $A(G)$. But $A(G)$ is a finitely generated abelian group according to the second part of Lemma \ref{L1}. Hence $y$ can be conjugate to $y^{p^n}$ for finitely many primes $p$ only, which concludes the proof.
  \end{proof}
%%%%%%%%%%%%%%%%%%%%%%%%%%%%%%%%%%%%%%%%%%%%%%%%%%
\section{Reduction to finitely generated groups}
%%%%%%%%%%%%%%%%%%%%%%%%%%%%%%%%%%%%%%%%%%%%%%%%%%
We first prove some lemmas which imply that in the proof of the Main Theorem
for $X(G)$, it suffices to consider the case of a finitely generated group $G$. First, notice that the map $X(H)\to X(G)$ induced by the inclusion of a subgroup $H<G$, is in general not injective (cf.\,Example \ref{subgroup}). However, the following lemma allows us to restrict our study to finitely generated subgroups.
\begin{lemm}\label{lift} Let $G$ be a group and $A\in M_n(R X(G))$ an idempotent
matrix, where $R$ is a commutative unital ring. Then 
there exist a finitely generated subgroup $H<G$ such that $A$
can be lifted to an idempotent $B\in M_n(R X(H))$.
\end{lemm}
In order to prove this lemma, we first need to see that Sidki doubles behave well under direct limits.
\begin{lemm}\label{fgcase}
Let $G$ be a group and denote by $\{G_\alpha\}_{\alpha\in I}$ the family
of its finitely generated subgroups. Let $R$ be a commutative unital ring.
Then there is a natural isomorphism
$$\varinjlim_{\alpha\in I} (X(G_\alpha)) \stackrel{\cong}\to X(G),$$
which induces an isomorphism $\varinjlim_{\alpha\in I} (M_n(R X(G_\alpha))) \stackrel{\cong}\to M_n(R X(G))$.
\end{lemm}
\begin{proof}
 For $G_\alpha <G_\beta$ we write $\iota_{\alpha\beta}$
 for the inclusion map $G_\alpha\to G_\beta$. By functoriality, there
 are natural (not
 necessarily injective) maps
 $$X_{\alpha\beta}:=X(\iota_{\alpha\beta}): X(G_{\alpha})\to X(G_\beta)\,,$$
 which on generators $x_\alpha,x_\alpha^\psi\in X(G_\alpha)$, $x_\alpha\in G_\alpha$ and $x_\alpha^\psi\in G_\alpha^\psi$, are given by $X_{\alpha\beta}(x_\alpha)= \iota_{\alpha\beta}(x_\alpha)\in X(G_\beta)$
 and $X_{\alpha\beta}(x^\psi_\alpha)=\iota_{\alpha\beta}(x_\alpha^\psi)$. The resulting natural map
$$X_* : \varinjlim _{\alpha\in I} X(G_{\alpha}) \longrightarrow X(G)$$
is  surjective, because an element $z$ in $X(G)$ involves only
finitely many generators $x, x^\psi\in X(G)$, and these can be viewed as elements in
some $X(G_\alpha)$, giving rise to an element $z_\alpha\in X(G_\alpha)$ which maps to $z\in X(G)$. 
To prove the injectivity, we take $z\in \varinjlim X(G_\alpha)$ which maps to $e$ in $X(G)$.
Write $K_\alpha$ for the kernel of $G_\alpha * G_\alpha \to X(G_\alpha)$. The element $z$ lifts to an element $z_\alpha\in G_\alpha * G_\alpha$. Writing $K$ for the kernel of the natural map $G * G\to X(G)$, we see that $z_\alpha$ maps to $K$ under the inclusion $G_\alpha *G_\alpha \to G * G$. Because $K$ is generated by conjugates of elements of the form $[g,g^\psi]$, we can find $\beta>\alpha$ such that $z_\alpha$ maps already to the kernel $K_\beta$ of
the natural map $G_\beta * G_\beta \to  X(G_\beta)$. We conclude that $z_\alpha$ maps to $e\in X(G_\beta)$. Thus its image  $z\in \varinjlim X(G_\alpha)$ is equal to $e$.

Because taking group rings and matrix rings commutes with direct limits over directed indexing
sets, we also have an isomorphism
$$M_n(R X_*): \varinjlim_{\alpha\in I} M_n(R X(G_\alpha))\longrightarrow 
 M_n(R X(G))\,.$$
\end{proof}
 \begin{rema}\label{lim} Just as the map $X(H)\to X(G)$ induced by a subgroup $H<G$ is in general not injective, the map $W(H)\to W(G)$ is also in general not injective (also Example \ref{subgroup}). However for $\{G_\alpha\}_{\alpha\in I}$
 the family of finitely generated subgroups of $G$ the natural map 
 $$\varinjlim _{\alpha\in I}W(G_\alpha)\to W(G)$$ 
 is an isomorphism. Indeed, the functor $\underset{\alpha\in I}{\varinjlim}$ is left exact and the exact sequences
 $$\{e\}\to W(G_\alpha)\to X(G_\alpha)\stackrel{\rho_\alpha}\rightarrow 
 G_\alpha\times G_\alpha \times G_\alpha\,$$
together with Lemma \ref{fgcase} allows us to conclude.
 \end{rema}
 \begin{proof}[Proof of Lemma \ref{lift}] Let $A^2=A=\{a_{ij}\}\in M_n(R X(G))$ with $a_{ij}=
 \sum_{x\in X(G)} a_{ij}(x)\cdot x$. Let
  $\{G_\alpha\}_{\alpha\in I}$
be the family of finitely generated subgroups of $G$. Because (cf. Lemma \ref{fgcase})
$$\varinjlim_{\alpha\in I} M_n(R  X(G_\alpha))\stackrel{\cong}{\rightarrow} 
M_n(R X(G))\,,$$
we can find an index $\beta\in I$ and a matrix $B_\beta \in M_n(R  X(G_\beta))$, 
such that $B_\beta$ maps to $A$. If $B_\beta$ is an idempotent, then $B=B_\beta$ proves the lemma. Otherwise $B_\beta$ is not an idempotent, so define 
$D=B_\beta^2- B_\beta$. Because $D$ maps to $0\in M_n(R  X(G))$
we can find
an index $\gamma\in I$, with $\gamma\geq\beta$, and such that the natural map
$$\theta_{\beta\gamma}: M_n(R X(G_\beta)) \rightarrow M_n(R X(G_\gamma))$$ 
yields $\theta_{\beta\gamma}(D)=0$. It follows that $B:=
\theta_{\beta\gamma}(B_\beta)$ is  an idempotent in $M_n(R X(G_\gamma))$
which maps to $A$.
 \end{proof}
 %%%%%%%%%%%%%%%%%%%%%%%%%%%%%%%%%%%
 \section{Proof of the Main Theorem}
 %%%%%%%%%%%%%%%%%%%%%%%%%%%%%%%%%%%%
 We have now all the ingredients to prove our main result.
 Note that if $X(G)$ is torsion-free then so is $G$, since $G$ is a retract of $X(G)$.
 Because for a torsion-free group, the idempotent conjecture is equivalent to the weak Bass Trace Conjecture for the case of $(1\times 1)$-matrices, the proof of our
 Main Theorem is a special case of the proof of the following result
 on idempotent $(n\times n)$-matrices.
 \begin{prop}\label{P}
 Let $G$ be a group and $X(G)$ its Sidki double. If $X(G)$ is torsion-free
 and $\C(G\times G\times G)$ satisfies the weak Bass Trace Conjecure,
 then $\C X(G)$ satisfies the weak Bass Trace Conjecture too.
 \end{prop}
 \begin{proof}
  Let 
 $A=(a_{ij})\in M_n(\C X(G))$ be an idempotent with $a_{ij}=\sum_{x\in X(G)} a_{ij}(x)\cdot x$. To prove the weak Bass Trace Conjecture for $A$ amounts
 to show
 that $\epsilon(A)=\kappa(A)$. This is equivalent to show that the sum 
 $\Sigma:=\sum_{[x]\in [X(G)\setminus \{e\}]}r_A(x)=0$. For
 this we write 
 $\Sigma=U + V$ where 
 $$U=\sum_{[x]\in\ [X(G)],\, \rho(x)\neq e}r_A(x), \,\quad \text{and}\quad V=\Sigma -U\,.$$
We first show that $U=0$.
 Write $B$ for the image of $A$ by the map
 $M_n(\C X(G))\to M_n( \C (G\times G\times G))$ induced by $\rho$. The sum $U$ defined above is also given by
$$W = \sum_{[q]\in [G\times G\times G], q\neq e}r_B(q)$$
and this sum equals zero, because $G\times G\times G$ is torsion-free and by assumption $B$
satisfies that $\epsilon(B)=\kappa(B)$. We conclude that $W=U=0$.

Next, we show that $V=0$: The idempotent $A$ can be lifted to an
 idempotent $D\in M_n(\C X(H))$ for some finitely generated subgroup $H<G$
  (cf.\,Lemma \ref{lift}). Write $D=(d_{ij})$ with
  $d_{ij}=\sum_h d_{ij}(h)\cdot h$, with $h\in X(H)$.
 If $v\in X(H)$ satisfies $\rho (v)=e$
 and is of infinite order, then 
 $r_D(v)=0$ by Proposition \ref{CCW}.
 Because $X(G)$ is assumed to be torsion-free, none of the terms $r_D(v)$ with
 $v\in X(H)$ of finite order, can contribute to $V$, because such a $v$ maps to $e$ in the torsion-free group $X(G)$. Thus, $V$ is a sum of terms of the form 
 $r_D(v)$ with $v$ of infinite order and these are all individually equal to 0.
 This shows that $V$ must be equal to $0$. 
 \end{proof}
The following variation of the Main Theorem and the Proposition holds, if we replace the
ground ring $\C$ by $\Z$. The assumption that $X(G)$ is torsion-free
can then be dropped, because for an arbitrary group $H$, 
$\Z H$ does not have
any idempotent besides of $0$ and $1$, since an idempotent $x\in \Z H$
obviously satisfies $\kappa(x)=x_e\in \Z$. 
\begin{prop} Let $G$ be a group and assume that $\Z(G\times G\times G)$ satisfies the weak Bass Trace Conjecture and let $X(G)$ be the Sidki double of $G$. Then
$\Z X(G)$ satisfies the weak Bass Trace Conjecture as well.
\end{prop}
\begin{proof}
We follow the line of proof of the Proposition \ref{P}. First, we 
observe that we can lift an idempotent matrix $A\in M_n(\Z X(G))$
 to an idempotent $B\in M_n(\Z X(K))$
for some finitely generated subgroup $K<G$ (cf. \ref{lift}).
As we can see $B$ as a matrix in $M_n(\Z X(K))$ we just need to
observe that elements $x\neq e$ in $X(K)$ of finite order have
$r_B(x)=0$ by Linnell's result {\cite{Linnell}*{Lemma 4.1}}.
Also, by assumption $G\times G\times G$ satisfy the weak Bass Conjecture
for $\Z (G\times G\times G)$. It follows that the image
of $\rho: X(G) \to G\times G\times G$ satisfies the weak Bass Trace
Conjecture as well. The rest of the proof is then as in the case of the Proposition \ref{P}.
 \end{proof}
 %%%%%%%%%%%%%%%%%%%%%%%%%%%%%%%%%%%%%%%%%%%
\section{Perfect groups and Stem-Extensions}
%%%%%%%%%%%%%%%%%%%%%%%%%%%%%%%%%%%%%%%%%%%%
A concise reference for this section is Kervaire's note
\cite{K}, and a more comprehensive treatment can be found in Gruenberg's book
\cite{G}*{Chapter 9, \S\! 9.9}. The article by Eckmann-Hilton-Stammbach
\cite{EHS} contains everything on stem-extensions we need here.
\begin{defi}
A central extension 
$C\to H\to Q$ is called a {\it stem-extension}, if $C<[H,H]$.
\end{defi}
%For the reminder of this Section, $Q$ will denote a perfect group. 
Central extensions ${\mathcal E}:  C\to H\to Q$ are classified by elements in $H^2(Q,C)$.
The {\sl{5-term sequence}} of $\mathcal{E}$ has the form
$$ H_2(H,\Z)\to H_2(Q,\Z)\stackrel{\partial{(\mathcal E)}}\longrightarrow H_1(C,\Z)\to H_1(H,\Z)\to
H_1(Q,\Z)\to \{e\}\,.$$
\begin{defi}A group $G$ is called {\it perfect} if $G=[G,G]$, or equivalently, $H_1(G,\Z)=\{e\}$. A group $G$ is called {\it super-perfect} if it is perfect and moreover $H_2(G,\Z)=\{e\}$.\end{defi}
For $Q$ perfect, the {\sl{Universal Coefficient Theorem}} implies that
$$H^2(Q, C)\cong \operatorname{Hom}(H_2(Q,\Z), C)$$
so that the central extension $\mathcal E$ corresponds to a homomorphism $\mathcal E _*: H_2(Q,\Z)\to C\,.$
The homomorphism $\mathcal E_*$ can be identified with $\partial(\mathcal E)$ in the 5-term sequence,
if one identifies $H_1(C,\Z)$ with $C$\,. 
\begin{lemm}[Eckmann-Hilton-Stammbach \cite{EHS}*{Prop. 4.1}]\label{equivalent} Let $\mathcal E: C\to G\to Q$ be a central extension with $Q$ a perfect group. Then the
following are equivalent:
\begin{enumerate}
\item[(1)] the classifying homomorphism $\mathcal E_*: H_2(Q, \Z)\to C$ is surjective ;
\item[(2)] $\mathcal E$ is a stem-extension;
\item[(3)] $G$ is perfect.
\end{enumerate}
\end{lemm}
\begin{proof} Because $Q$ is perfect, the 5-term sequence for $\mathcal E$
ends with
$$H_2(Q,\Z)\stackrel{\partial{(\mathcal E)}}\longrightarrow C \to G/[G,G]\to \{e\}\,.$$
Identifying $\mathcal E_*$ with $\partial(\mathcal E)$ we see that $(1)\Longleftrightarrow (3)$.
If $G$ is perfect, $C<[G,G]=G$, so $(3) \Rightarrow (2)$. Since $Q$ is 
assumed to be perfect,
$G=C[G,G]$ and assuming (2), i.e.,\, $C<[G,G]$, we have $G=[G,G]$.
Thus $(2)\Rightarrow (3)$.
\end{proof}
The {\sl{universal}} stem-extension $\mathcal U$ of a perfect group 
$Q$ is obtained by choosing 
$C=H_2(Q,\Z)$ and for $\mathcal U_*: H_2(Q,\Z)\to C$ the identity homomorphism, yielding
 $$\mathcal U:  H_2(Q,\Z)\longrightarrow \tilde{H}\longrightarrow Q\,.$$ 
It follows from Lemma \ref{equivalent},
 $(1)\Rightarrow (2)$, that
  $\mathcal U$ is indeed a stem-extension, i.e. that $H_2(Q,\Z)$ is contained in ${[\tilde{H},\tilde{H}]}$. Moreover, the universal stem-extension has the following property.
\begin{lemm}[Kervaire \cite{K}*{Proposition 1}]\label{K}
Let $Q$ be a perfect group and
 $$\mathcal E: C\longrightarrow H \longrightarrow Q$$
 a stem-extension. If
$\mathcal U: H_2(Q,\Z)\to \tilde{H}\to Q$ denotes the universal stem-extension,
then $\tilde{H}$ is super-perfect and $\mathcal U$ maps onto $\mathcal E$ to yield a commutative diagram
 $$\xymatrix
 { H_2(Q,\Z) \ar@{->>}[d]^{\partial(\mathcal E)}\ar[r]& \tilde{H} \ar@{->>}[d]
 \ar[r] &Q\ar@{=}[d]\\
C\ar[r]&H\ar[r] & Q.\\
 }
 $$ 
\end{lemm}
For $G$ a perfect group, the Sidki Double $X(G)$ is particularly simple to
describe.
\begin{prop}\label{Sstem}
 Let $G$ be a perfect group. Then 
 $$W(G)\to X(G)\stackrel{\rho}\longrightarrow G\times G\times G$$
 is a stem-extension, i.e., $\rho$ is surjective  and
 $W(G)$ is a central subgroup of $X(G)$ contained in $[X(G),X(G)]$.
  \end{prop}
 \begin{proof}
 First we check that $\rho: X(G)\to G\times G\times G$ is surjective.
 It is known that $\rho$ maps the subgroup $D(G)L(G)<X(G)$ to a subgroup
 of $G\times G\times G$ which contains the commutator subgroup of
  $G\times G\times G$ (see Proof of Lemma \ref{L1} above).  Because $G$
  and therefore also $G\times G\times G$ is perfect, this implies that 
 $$\rho(L(G)D(G)) = \rho(X(G))=G\times G\times G$$
 and in particular $\rho$
 is surjective. Next, we check that $W(G)<X(G)$ is central. We know already that
 $W(G)$ is central as a subgroup of $L(G)D(G)$, so it suffices to see that 
 $X(G)=L(G)D(G)$. Because
 the kernel of $\rho$ lies in $L(G)D(G)$
 it follows that 
 $$X(G)/L(G)D(G)\simeq \rho(X(G))/\rho(L(G)D(G))=\{e\}$$  
 and thus $L(G)D(G)=X(G)$. To see that $W(G)<[X(G),X(G)]$, we observe that
 as $W(G)<D(G)$ and $D(G)=[G,G^\psi]<[X(G),X(G)]$.
   \end{proof}
\begin{coro}
Let $G$ be a super-perfect group. Then
$$X(G)\cong G\times G\times G\,.$$
\end{coro}
\begin{proof}
The assumption on $G$ being perfect implies that $G\times G\times G$ is a perfect group.
According to Proposition \ref{Sstem}, the extension
$$W(G)\to X(G)\to G\times G\times G$$ 
is a stem-extension. The associated
universal stem-extension has the form
$$H_2(G\times G\times G,\Z)\longrightarrow \widetilde{X(G)}\longrightarrow G\times G\times G\,.$$
From Lemma \ref{K} we see that  there is a surjective homomorphism 
$$H_2(G\times G\times G,\Z)\longrightarrow W(G)\,.$$
But as $G$ is super-perfect, the $K\ddot{u}nneth$-Formula shows that
$G\times G$ as well as $G\times G\times G$ are super-perfect too. Thus $H_2(G\times G\times G,\Z)=\{e\}$ and therefore $W(G)=\{e\}$.
\end{proof}
\section{Examples}\label{examples}
General theorems on the structure of Sidki Doubles are difficult to obtain.
However, it is easy to see that the Sidki Double of an amenable group is amenable using the map $\rho: X(G)\to G\times G\times G$ and observing that both the kernel $W(G)$ of $\rho$ and the image $\rho(X(G))$ are
amenable. But, for instance,  we don't know whether for 
an a-T-menable group $G$ its Sidki Double is a-T-menable as well.
We give an example of groups $H<G$ for which the induced map of the Sidki doubles $X(H)\to X(G)$
fails to be injective (Example \ref{subgroup}) and we show that for $G$ of finite
cohomological dimension, $X(G)$ can have infinite rational
cohomological dimension (Example \ref{cd}). We also give an example
of a torsion-free finitely presented group $G$ with Sidki Double $X(G)$ admitting a finitely generated non-free, projective $\Q X(G)$-module for 
$\Q$ the rational numbers (Example \ref{nonfreeproj}).

\medskip

In Section 2 we proved results on $X(G)$ by passing to $X(H)$ for a finitely generated subgroup $H<G$; the induced map $X(H)\to X(G)$ will not be
injective in general  as the following example shows.
 \begin{exam}\label{subgroup}
 There is a finitely presented group $G$ and a finitely generated subgroup $H<G$
 such that the natural maps $X(H)\to X(G)$ and $W(H)\to W(G)$
 are not injective.
 \end{exam}
 \begin{proof}
 It was proved in \cite{KS}*{Lemma 9.1} that for $\Gamma=\Z\times \Z$,
 there is a natural isomorphism $W(\Gamma)\cong H_2(\Gamma,\Z)=\Z$. Therefore,
 by writing $\Q\times\Q$ as a direct limit of groups $\Z\times\Z$
and using Lemma \ref{lim}, we see that $W(\Q\times\Q)\cong\Q$. 
One can embed $\Q$
into a finitely presented group $T$ which is simple and has type $FP_\infty$ 
(cf.\,Belk, Hyde and Matucci \cite{BHM}*{Theorem 3 and 4}). Thus $\Q\times\Q<T\times T$ with $T\times T$
perfect and of type $FP_\infty$. By Theorem B of \cite{KS} it follows that
 $W(T\times T)$ is a finitely generated abelian group. We conclude that $W(\Q\times \Q)\cong \Q$
cannot map injectively to $W(T\times T)$.
Now $\Q\times\Q$ is the union
of subgroups $\{S_\alpha, \alpha\in I\}$ with every $S_\alpha$
isomorphic to $\Z\times\Z$. Therefore we can find an  index $\beta\in I$ and
$S_\beta<\Q\times\Q$ such that the map $W(S_\beta)\to W(T\times T)$ induced by $S_\beta\to \Q\times\Q\to T\times T$ is not injective either.  Taking 
$$H:=S_\beta<\Q\times\Q<T\times T=:G$$ yields the desired example.
 \end{proof}
 Bridson and Kochloukova asked in \cite{BK}*{Question 5.3}, whether
for a finitely generated group $F$ of cohomological dimension one, $X(F)$ has finite cohomological dimension. We give
an example of a group $G$ with cohomological dimension 2
but with $X(G)$ having infinite rational cohomological dimension.
 \begin{exam}\label{cd}
 There is a group $G$ of cohomological dimension 2 whose
 Sidki Double has infinite rational cohomological dimension.
 \end{exam}
 \begin{proof}
 Take a group $\Sigma^2$ with classifying space $K(\Sigma^2,1)$ a finite 2-dimensional
 $CW$-complex with the same homology as the 2-sphere $S^2$ (for the construction of such a group see for instance  Maunder \cite{M}).
Put $G=\ast_\N\Sigma^2$ a countably infinite free product of groups $\Sigma^2$.
$G$ has cohomological dimension 2 and $H_2(G,\Z)=\oplus_\N \Z $. According to \cite{KS} (see also \cite{Rocco}), the abelian group $W(G)$ maps onto $H_2(G,\Z)$ and we see that $W(G)$
and hence
$X(G)$ contains a free abelian subgroup of infinite rank. We conclude that
the rational cohomological dimension of $G$ must be infinite. 
\end{proof}
 The final example concerns a group $G$ with torsion-free Sidki Double $X(G)$,
 admitting a finitely generated projective $\Q X(G)$-module which is not free. For a group homomorphism $f:K\to L$, we write
$f_*$ for the functor from ${\Q}K$-modules to ${\Q}L$-modules defined by $f_*(M)= {\Q}L\otimes_{{\Q}K}M$, noting that the functor $f_*$ maps projectives to projectives.
 \begin{exam}\label{nonfreeproj} Let $G$ be the Baumslag-Solitar group $BS(1,n)$ with $n>1$.
 Then $X(G)$ is torsion-free and admits a projective, non-free $\Q X(G)$-module $\overline{P}$
 satisfying $\epsilon(\overline{P})=\kappa(\overline{P})=1$.
 \end{exam}
 \begin{proof}
 Levin proved that for $n>1$,  $G=BS(1,n)$ admits a rank 1 projective, non-free
 $\Q G$-module $P$, see \cite{Lewin}*{Example, p. 63}.
 Also $W(BS(1,n))=0$ for $n>1$, see \cite{KS}*{Lemma 9.1,(3)} and
 it follows that $X(BS(1,n))$ is torsion-free. 
There is an injection $\sigma: G\to X(G)$, with retraction $\tau: X(G)\to G$.
Because $\tau_* \sigma_*(P)=P$ we conclude that $\overline{P}=\sigma_* P=\Q X(G)\otimes_{\Q G} P$ is a projective, non-free
 module over $\Q X(G)$. Because $P\oplus \Q G\cong \Q G\oplus \Q G$
 (cf.\,\cite{Lewin}), we also have $\sigma_*(P)\oplus \Q X(G)\cong
 \Q X(G)\oplus \Q X(G)$ and thus
 $\epsilon(P)=\kappa(P)=\epsilon(\overline{P})=\kappa(\overline{P})=1$.
 \end{proof}

We now explain
how to construct torsion-free geometrically finite groups $G$ (groups admitting a $K(G,1)$ which is homotopy equivalent to a finite CW-complex), such that $X(G)$
contains non-trivial elements of finite order.
\begin{prop}\label{torsion}
Let $G$ be a perfect group.  
 If $H_2(G,\Z)$ is a torsion group and contains an element of order $n>1$,
then $X(G)$ contains an element of order $n$ too.
    \end{prop}
 \begin{proof}
 We observe that for a perfect group $G$, the $K\ddot{u}nneth$-Formula
 implies that
 $$H_2(G\times G\times G,\Z)\cong H_2(G,\Z)\oplus H_2(G,\Z)\oplus H_2(G,\Z)\,.$$ 
 Because $W(G)\to X(G)\to G\times G\times G$ is a stem-extension, we see from Lemma \ref{K}, that the universal stem-extension of
 $G\times G\times G$ yields a surjective homomorphism
 $$H_2(G,\Z)\oplus H_2(G,\Z)\oplus H_2(G,\Z)\to W(G)\,.$$
 As $H_2(G,\Z)$ is a torsion group, this implies that $W(G)$ is a torsion group too.
 On the other hand, it is known
 that $W(G)$ maps onto $H_2(G,\Z)$, see for instance \cite{Rocco} Lemma 2.2. It follows that if $H_2(G,\Z)$ contains 
 an element $x$
 of order $n$ and $y\in W(G)$ is a counter-image of $x$, then $y$ is a torsion 
 element and some multiple of $y$ will have order $n$.
\end{proof}
 We now can construct concrete examples of torsion-free groups with
 torsion in their Sidki doubles.
 \begin{prop}\label{TorsionExample} Let $n>1$ be an integer.
 There exists a finitely presented torsion-free group $G(n)$ with
 classifying space $K(G(n),1)$ a finite $CW$-complex such that
 $X(G(n))$ contains an element of order $n$. Moreover, the group $G(n)$ can be chosen so that $\C G(n)$ does not contain any non-trivial idempotent.
 \end{prop}
 \begin{proof}
According to Leary \cite{Leary}*{Theorem A}, for any $n\in\N$ there is $G(n)$ a cubical CAT(0)-group with $K(G(n),1)$ a finite CW-complex with the same
homology as the mapping cone $C(\phi)$ of a degree $n$ map $\phi: S^2\to S^2$; 
 ($C(\phi)$ is a Moore space $M(\Z/n\Z,2)$).
  Then $G(n)$ is perfect and $H_2(G(n),\Z)\cong \Z/n\Z$.
It follows from Proposition \ref{torsion} that $X(G(n))$ contains an element of order $n$. 
Because cubical CAT(0)-groups satisfy the Bass Conjecture and $G(n)$ is torsion-free,
$\C G(n)$ does not have any non-trivial idempotent.
  \end{proof}
\begin{rema} We can also follow Maunder \cite{M} and take for $G(n)$ a group with $K(G(n),1)$ a finite
 CW-complex of dimension three, with the same homology as the mapping cone $C(\phi)$ of a degree $n$ map $\phi: S^2\to S^2$. This would also give a torsion-free group $G(n)$ with torsion in $X(G(n))$, but then we don't know if $\C G(n)$ is known to not have any non-trivial idempotent.
\end{rema}

\end{document}